\documentclass[11pt,english]{article}
\usepackage[margin=1.95 cm,bottom=15mm,footskip=7mm,top=15mm]{geometry}
\usepackage{enumerate}
\usepackage{url}
\usepackage{amssymb}
\usepackage{amsthm}
\usepackage{amsmath}
\usepackage{enumerate}
\usepackage{boolexpr, xstring}
\usepackage{enumitem}
\usepackage[numbers]{natbib}
\usepackage{subcaption}
\usepackage{floatrow}
\usepackage{setspace}
\usepackage[colorlinks=false, hidelinks]{hyperref}
\usepackage{appendix}
\usepackage[nameinlink,capitalise,noabbrev]{cleveref}
\crefname{appsec}{Appendix}{Appendices}
\crefformat{equation}{#2(#1)#3}

\theoremstyle{plain}

\newtheorem{theorem}{Theorem}[section]
\newtheorem{proposition}[theorem]{Proposition}
\newtheorem{lemma}[theorem]{Lemma}

\newtheorem*{question*}{Question} \Crefname{question}{Question}{Questions}

\theoremstyle{definition}

\newtheorem{question}[theorem]{Question}

\theoremstyle{remark}

\newcommand{\dist}{\operatorname{dist}}

\begin{document}

\title{\texorpdfstring{\vspace{-1.5cm}}{}
Bounded Degree Spanners of the Hypercube}
\date{}
\author{
Rajko Nenadov\thanks{Department of Mathematics, ETH, Z\"urich, Switzerland. }
\and
Mehtaab Sawhney\thanks{Massachusetts Institute of Technology, Cambridge, MA 02139, USA. Email:
\href{mailto:msawhney@mit.edu} {\nolinkurl{msawhney@mit.edu}}.
}
\and
Benny Sudakov\thanks{Department of Mathematics, ETH, Z\"urich, Switzerland. Email:
\href{mailto:benjamin.sudakov@math.ethz.ch} {\nolinkurl{benjamin.sudakov@math.ethz.ch}}.
Research supported in part by SNSF grant 200021-175573.}
\and
Adam Zsolt Wagner\thanks{Department of Mathematics, ETH, Z\"urich, Switzerland. Email:
\href{mailto:zsolt.wagner@math.ethz.ch} {\nolinkurl{zsolt.wagner@math.ethz.ch}}.}
}
\maketitle
\begin{abstract}
    In this short note we study two questions about the existence of subgraphs of the hypercube $Q_n$ with certain properties. The first question, due to Erd\H{o}s--Hamburger--Pippert--Weakley, asks whether there exists a bounded degree subgraph of $Q_n$ which has diameter $n$. We answer this question  by giving an explicit construction of such a subgraph with maximum degree at most 120.
    
    The second problem concerns properties of $k$-additive spanners of the hypercube, that is, subgraphs of $Q_n$ in which the distance between any two vertices is at most $k$ larger than in $Q_n$. Denoting by $\Delta_{k,\infty}(n)$ the minimum possible maximum degree of a $k$-additive spanner of $Q_n$, Arizumi--Hamburger--Kostochka showed that
    $$\frac{n}{\ln n}e^{-4k}\leq \Delta_{2k,\infty}(n)\leq 20\frac{n}{\ln n}\ln \ln n.$$ We improve their upper bound by showing that 
    $$\Delta_{2k,\infty}(n)\leq 10^{4k} \frac{n}{\ln n}\ln^{(k+1)}n,$$where the last term denotes a $k+1$-fold iterated logarithm.
\end{abstract}

\section{Introduction}
Let $Q_n$ denote the hypercube graph, with vertex set $\{0,1\}^n$ with edges connecting two vertices if they differ in precisely one coordinate. Sparse subgraphs of the hypercube with strong distance-preserving properties have been studied extensively in the literature, and have found many practical applications in distributed computing and communication networks. We refer the reader to the recent survey~\cite{spannersurvey}.

Erd\H{o}s--Hamburger--Pippert--Weakley~\cite{Erdos} studied spanning subgraphs of $Q_n$ with diameter $n$. They observed that there exists such a subgraph with average degree $2+O\left(\frac{1}{\sqrt{n}}\right)$, however in their construction there were vertices of degree $n$. They asked the following natural question:
\begin{question}[Erd\H{o}s--Hamburger--Pippert--Weakley~\cite{Erdos}]\label{ques:erdos}
Does there exist a spanning subgraph of $Q_n$ with bounded degree and diameter $n$?
\end{question}
Our first result is an explicit construction giving a positive answer to~\cref{ques:erdos}.
\begin{theorem}\label{thm:main}
There exists a spanning subgraph $G$ of $Q_n$ with maximum degree at most $120$ such that the diameter of $G$ is $n$.
\end{theorem}

One particular distance-preserving property that has received much attention in the past is that of an \emph{additive spanner}. We say a subgraph $G\subset Q_n$ is a \emph{$k$-additive spanner} if $\dist_{G}(x,y)\leq \dist_{Q_n}(x,y)+k$ for any two vertices $x,y\in\{0,1\}^n$. Constructions of additive spanners with few edges and/or low maximum degree have attracted considerable attention in computer science in the past, see e.g.~\cite{add5,add1,add6,add3,add2,add4,add7} and the many references therein.

Arizumi--Hamburger--Kostochka~\cite{AHK} denoted by $\Delta_{k,\infty}(n)$ the minimum possible maximum degree of a $k$-additive spanner of $Q_n$. Note that since $Q_n$ is bipartite, by deleting edges the distance can only grow by an even amount. They showed that for $k\geq 2$ and $n\geq 21$ we have $$\frac{n}{\ln n}e^{-4k}\leq \Delta_{2k,\infty}(n)\leq 20\frac{n}{\ln n}\ln \ln n.$$ Their lower bound is a short argument given by counting the vertices of a certain distance from a fixed vertex, and their upper bound is an explicit construction. Our second result is an improvement of their upper bound on this problem.
\begin{theorem}\label{thm:main2}
For all $n$ sufficiently large and $k\in \mathbb{N}$, there exists a $2k$-additive spanner of the hypercube with maximum degree at most  
\[10^{4k}\frac{n\ln^{(k+1)} n}{\ln n}.\]
\end{theorem}
Note here that $\ln^{(k+1)} n$ is the $k+1$-times iterated logarithm, defined by $\ln^{(1)} n=\ln n$ and $\ln^{j+1} n=\ln\left(\ln^{(j)} n\right)$.

We prove~\cref{thm:main} in~\cref{sec:mainthm} and prove~\cref{thm:main2} in~\cref{sec:mainthm2}. Some open questions and further directions of study are given in~\cref{sec:concl}.

\section{Bounded degree subgraph preserving diameter}\label{sec:mainthm}

In the present paper, a \emph{perfect code} will always mean a perfect 1-error-correcting code over the alphabet $\{0,1\}$ with codewords having length $n$. We say that $\mathcal{C}$ is a perfect $1$-error-correcting code if any two codewords have Hamming distance at least three, and moreover the radius one Hamming balls centered on the codewords partition the whole space $\{0,1\}^n$. Note that the number of codewords in a perfect code $\mathcal{C}$ is $|\mathcal{C}|=\frac{2^n}{n+1}$. Perfect codes exist whenever $n=2^r-1$ for some $r\in\mathbb{N}$, see e.g.~\cite{perfectcodes}. We will use the fact that for all $n=2^r-1$, $r\in\mathbb{N}$, it is possible to partition the space $\{0,1\}^n$ into $n+1$ perfect codes (see e.g.~\cite{perfectcodes}, p. 15.).

In this section we prove~\cref{thm:main}. We first need the following technical lemma.
\begin{lemma}
For all $n$ there is a subset $S$ of vertices of $Q_{n}$ with the following two properties. 
\begin{itemize}
    \item Every vertex of $v$ is either in $S$ or adjacent to a member of $S$.
    \item Every vertex of $v$ is adjacent to at most $2$ vertices in $S$.
\end{itemize}
We refer to such a subset of vertices $S$  of $Q_n$ as a \emph{nearly perfect code}.
\end{lemma}
\begin{proof}
Note that for $n = 2^{k}-1$ the result follows from the existence of perfect codes. For other values of $n$ let $k$ be such that $n$ is between $2^{k}-1$ and $2^{k+1}-1$ and divide the $n$ coordinates into $2^{k}-1$ buckets of size at most $2$. Now take a perfect code $\mathcal{C}$ of $Q_{2^{k}-1}$. We define $S$ by taking the sum in $\mathbb{F}_2$ of each bucket and take the points whose image under this transformation is in the perfect code of $Q_{2^{k}-1}$. It is now straightforward to verify that $S$ satisfies the necessary conditions. Indeed, given any $w\in Q_n$, consider the word $u\in Q_{2^k-1}$ obtained from by by taking the sum in $\mathbb{F}_2$ of each bucket. Then $u$ is either in $\mathcal{C}$ or adjacent to some $u'$ in $\mathcal{C}$. In the first case $w\in S$ and $w$ is not adjacent to any other element of $S$. In the second case $w$ is adjacent to exactly all those $w'\in \mathcal{C}$ which agree with $w$ everywhere except on the block which corresponds to the bit in which $u$ and $u'$ differ. Moreover on this block $w'$ has different parity from $w$. Since all blocks have length at most 2, we conclude that $w$ is adjacent to at most two members of $S$.
\end{proof}

We now proceed first by proving a weaker version \cref{thm:main} which only requires preserving the distance between antipodes. We say two vertices are antipodal if they differ on all $n$ coordinates.

\begin{lemma}\label{lem:init}
There exists a subgraph $H$ of $Q_n$ with max degree $10$ such that all antipodes are at distance $n$ within $H$. 
\end{lemma}
\begin{proof}
Consider the set of vertices $(0,\ldots,0)$, $(1,0,\ldots,0)$, $(1,1,0,\ldots, 0)$, $\ldots$, $(1,\ldots,1)$, $(0,1,\ldots,1)$, $(0,0,1,\ldots,1)$, $\ldots$, and $(0,\ldots,0,1)$, i.e.~all vertices with coordinates having all $0$'s and then followed by $1$'s  or having all $1$'s and then followed by $0$'s . Note that these points form a cycle $C$ of length $2n$ in $Q_n$ such that there are $n$ pairs of antipodes along this cycle. In particular, if a vertex is on $C$ then there is a path of length $n$ from this vertex to its antipode using only edges of $C$.

The construction of $H$ is to translate this $2n$-cycle $C$ by the appropriate nearly perfect code so that every vertex is contained in one of these cycles. By the above discussion, if some vertex is on a translate of $C$ then so is its antipode and therefore their distance in $H$ is $n$.

Let $(e_i)_{i=1}^n$ be the standard basis vectors in $\mathbb{F}_2^{n}$ and for $1\leq k \leq n$ define $f_k$ as  
\[f_k = \sum_{i=1}^{k} e_i.\] Note that the vectors $f_i$ form a basis of $\mathbb{F}_2^{n}$ and that the vertices of the cycle $C$ are exactly $\textbf{0}, f_1,\ldots, f_n,f_n-f_1, f_n-f_2,\ldots,f_n-f_{n-1}$. Now consider a nearly perfect code $S'$ in the basis of the $f_i$ vectors and all the translates $s+C$ where $s\in S'$. First note that every element is contained in at least one cycle as we have translated $0$ by a nearly perfect code in the $f_i$ basis and all basis vectors $f_1,\ldots,f_n$ are in the cycle. To see that no vertex however is in more than five cycles consider the element $s\in S'$ so that $v\in s+C$. Then it follows that 
$v = s, s+ f_j,$ or $s+f_n-f_j$ for some $j$ and therefore $s = v, v-f_j$, or  $v-f_n+f_j$. Thus $s$ is either $v$, a  neighbor of $v$, or a neighbor of $v-f_n$ and by the definition of nearly perfect codes we have that every vertex is either in the code or next to at most two other code words we have that every vertex is in at most $5$ cycles and hence has degree at most $10$.
\end{proof}
We now use \cref{lem:init} to complete the proof of \cref{thm:main}; note that the fact that the construction comes from a union of these antipodal cycles plays a critical role in \cref{thm:main}.
\begin{proof}
For $n\le 100$ note that taking $Q_n$ suffices. Otherwise define $n_i$ such that $n = \sum_{i=1}^{4} n_i$, $n_i = \lfloor \frac{n}{4}\rfloor$ or $n_i = \lceil \frac{n}{4}\rceil$, and $n_i\ge n_j$ for $i\le j$. Now associate $Q_n$ with its representation on $\{0,1\}^{n}$ and divide the coordinates in blocks $B_1,B_2,B_3,B_4$ with $B_i$ having block size $n_i$. Finally for each vertex $v$ define $v_i$ to be restriction of $v$ to the block $B_i$.

We first assign a vertex $v$ its neighbors based only on its restrictions $v_i$. Consider $H_i$ that is a subgraph of $Q_{n_i}$ coming from \cref{lem:init}. Now consider an almost equal partition, $M_i$, of $[n]\setminus B_i$ into $n_i$ parts. As $n\geq 100$, note that every part has size at most four. Now as one walks along each of the antipodal cycles of size $2n_i$ in $H_i$ assign to each vertex a part of $M_i$ such that any path between antipodes covers all parts of the partition $M_i$. (In order to do so simply order the partition $M_i$ and assign its parts cyclically along the cycles of $H_i$.) We now define the neighbors of $v$ by considering its restriction to $v_i$ for each $i$ and keeping the neighbors in the directions which $v_i$ has in $H_i$ plus the neighbors in the directions of the part of the partition $M_i$ of $[n]\setminus B_i$ to which it was assigned. This is the  desired subgraph $G$ if the hypercube.  

We first verify that the max degree of $G$ is at most $120$. First note that since the $H_i$ have max degree at most $10$ and we have four parts the contribution to each vertex $v's$ degree from the $H_i$ is bounded by $40$. Furthermore since each part of $M_i$ has size at most four, there are at most four edges outgoing from $v$ due to this part. Since every vertex was on at most five cycles in $H_i$, the total number of edges outgoing from $v$ due to the partition of $[n]\setminus B_i$ into $n_i$ parts is at most $5\cdot 4 = 20$. The only subtlety is now  to account for edges incoming to $v$ from each of the partitions of $[n]/B_i$. But note that for such edges $v$ to $w$ we have $v_i = w_i$ and therefore this relationship is in fact symmetric unlike the asymmetric description. Therefore the total degree count is $40 + 4(20) = 120$.

Finally we demonstrate that $G$  is diameter $n$. Consider two vertices $v$ and $w$ such that $v_i$  and $w_i$ match on  exactly $k_i$ coordinates. Suppose that $k_1\le k_2\le k_3\le k_4$; the other cases are handled in an analogous manner. In order to ``fix" $v$ we first proceed along the cycle in $H_1$ to $v_1$'s antipode in order to obtain $x$. Along the way we can adjust the coordinates in $B_3$ and $B_4$ so that $x_3 = w_3$ and $x_4 = w_4$ while $x_2$ is the antipode of $w_2$. Note that we take $n_1$ steps along $H_1$, $k_2$ to change the coordinates in $B_2$ appropriately, and $n_3-k_3+n_4-k_4$ steps to fix $B_3$ and $B_4$. We now walk from $x_2$ to $w_2$ along $H_2$ and along the way fix $H_1$ so that it now matches $x_1$. This takes $k_1$ steps to fix $B_1$ to $w_1$ and $n_2$ to fix $B_2$ to $w_2$. Therefore the total number of steps is 
\[n_1+k_2+n_3-k_3+n_4-k_4+k_1+n_2\]
\[ = (n_1+n_2+n_3+n_4)+(k_1+k_2-k_3-k_4)\]\[\le (n_1+n_2+n_3+n_4) = n\]
as desired.
 
\end{proof}


\section{$k$-additive Spanner of the Hypercube}\label{sec:mainthm2}
The main goal of this section is to prove~\cref{thm:main2} by constructing a $2k$-additive spanner of the hypercube with small maximum degree. For the sake of clarity various floor and ceiling symbols will be omitted. The key idea is to essentially iterate the construction in~\cite{AHK} which achieves this result for the $k=1$ case (with a slightly better constant). 
\begin{proof}
We build the construction iteratively as $k$ increases. We will maintain the following invariant for the $2k$-additive spanners: for any pair of points which are distance $\ell$ apart there is a path connecting them of length at most $\ell+2k$ whose vertices have at least $\frac{\ell}{32^{k+1}}$ different coordinate sums. For $k=0$ the construction is taking the entire hypercube graph $Q_n$. This satisfies the necessary maximum degree condition and will serve as the base case for this induction on $k$. Note that in $Q_n$ between any two points at distance $\ell$ we may first flip all necessary zero coordinates to ones and then all required ones to zeros, ensuring the existence of a path of length $\ell$ whose vertices have at least $\ell/2$ different coordinate sums.  For the remainder of the proof we divide the coordinates into two groups. 
\begin{itemize}
\item Pick $r\in \mathbb{N}$ so that $2^{r}-1\in [\sqrt{n}/2,\sqrt{n}]$, and let $B_0$ be the first $q = 2^{r}-1$ coordinates.
\item $B_1,\ldots,B_t$ will be an almost equal partition of the remaining $n-q$ coordinates into $t = \frac{\ln n\cdot \ln^{(k)}n}{900\left(\ln^{(k+1)}n\right)^2}$ blocks. 
\end{itemize}
We define an additional parameter $s$ to be $s = \frac{\ln n}{10\ln^{(k+1)}n}.$ The construction now has three distinct parts.
\begin{itemize}
    \item Define $H_1$ to be subgraph created by including all edges in directions in $B_0$ for every vertex in $Q_n$. That is, two adjacent vertices $x,y\in Q_n$ are connected by an edge in $H_1$ precisely if they differ in only one of the first $q$ coordinates and nowhere else. Note that $H_1$ is a vertex-disjoint union of $2^{n-q}$ copies of $Q_{q}$.
    \item We now define the subgraph $H_2$. Since $q=2^{r}-1$, as remarked at the beginning of~\cref{sec:mainthm}, we can partition each disjoint copy of $Q_{q}$ in $H_1$ into perfect codes $D'_1,\ldots,D'_{q+1}$ and let $D_i$ be the union of the $D_{i}'$ over these disjoint components of $H_1$. Now  fix a bijective map $f$ from $\{D_1,D_2,\ldots,D_{\binom{t}{s}}\}$ to $\binom{[t]}{s}$. Define the subgraph $H_2$ as follows. For every vertex $x$, first find the index $i$ so that $x$ belongs to $D_i$. If $i\leq \binom{t}{s}$ then let $\{B_{i_1}, B_{i_2},\ldots,B_{i_s}\}$ be the set of $s$ blocks of coordinates from $B_1, \ldots, B_t$ corresponding to $D_i$. Next, let $B_x:=B_{i_1}\cup\ldots\cup B_{i_s}$, fix all coordinates of $x$ on $[n]\setminus B_x$, and on the coordinates in $B_x$ include the $2(k-1)$-additive spanner on $|B_x|$ many coordinates that is given by the induction hypothesis. For example, if $x$ belongs to $D_5$ and $f(D_5)=\{1,2,\ldots,s\}$ then we fix the coordinates of $x$ outside of $B_1\cup\ldots\cup B_s$ and include the $2(k-1)$-spanner construction given by the induction hypothesis on the approximately $(n-q)s/t$ coordinates in $B_1\cup\ldots\cup B_s$.
    \item We now define $H_3$. First divide $B_1,\ldots,B_t$ into groups of size $j=\frac{500^kt}{s}$; label these $A_1,\ldots,A_{s/500^k}$. Now for each vertex $x$ which belongs to a $D_i$ with $i\leq \binom{t}{s}$, let $B_x:=B_{i_1}\cup\ldots\cup B_{i_s}$ where $\{i_1,\ldots,i_s\}=f(D_i)$ as above. Next, let $s_x:=\sum_{i\in B_x}x_i$. Take this coordinate sum $s_x$ mod $s/500^k$, call this $s'$, and for the vertex $x$ only include edges in directions $A_{s'}\setminus B_{i_1}\cup\ldots\cup B_{i_s}$. 
\end{itemize}
Note that in $H_2$ and $H_3$ the edges are (implicitly) ``directed'' from one vertex to another. However one can verify that the definitions are symmetric in both  cases. First consider $H_2$. Note that  if a vertex $y$ agrees with $x$ on every coordinate outside of $B_x$ then in particular they agree on the first $q$ coordinates used to define the sets $D_i$. Therefore we have that $x$ and $y$ belong to the same $D_i$ and thus $B_x=B_y$. Hence in the construction of $H_2$ when we  consider  $y$ we include the same $2(k-1)$-spanner construction on $B_x$ as we did when considering the vertex $x$. 

Next consider $H_3$. If $x$ is connected to $y$ in $H_3$ this implies that the coordinate on which $y$ differ from $x$ lies outside of $B_x$ and also outside of the set of the first $q$  coordinates. Therefore, both $x$ and $y$ belong to the same set $D_i$, $B_x=B_y$, and also $s_x=s_y$. So in both cases the value of $s'$ is the same, hence in the construction of $H_3$ we included edges touching $x$ and $y$ in precisely the same directions.

The desired subgraph of $Q_n$ is simply $H = H_1\cup H_2\cup H_3$. We first verify that the subgraph $H_2$ is well defined in that the desired bijection $f$ indeed exists.
\begin{lemma}
For $n$ sufficiently large, we have
\[\binom{t}{s}\le \sqrt{n}/2.\]
\end{lemma}
\begin{proof}
Note that 
\[\binom{t}{s}\le \bigg(\frac{te}{s}\bigg)^{s}= \exp\bigg(s\bigg(\ln\bigg(\frac{te}{s}\bigg)\bigg)\le \exp\bigg(\frac{\ln(n)}{9\ln^{(k+1)}(n)}\ln(\ln^{(k)}(n))\bigg)\le \sqrt{n}/2\]
for $n$ sufficiently large as desired.
\end{proof}
We now prove the desired bound on the maximum degree of the graph $H=H_1\cup H_2 \cup H_3$.
\begin{lemma}
The maximum degree of $H$ is at most \[10^{4k}\frac{n\ln^{(k+1)}(n)}{\ln(n)}\] for all $k$ and for $n$ sufficiently large. 
\end{lemma}
\begin{proof}
We proceed by induction on $k$. Note that for $k=0$ this statement is trivial. For larger $k$ note that the maximum degree in $H_1$ is at most $\sqrt{n}$, the maximum degree in $H_2$ is upper bounded, using the induction hypothesis for the $2(k-1)$-spanner on $(n-q)\cdot\frac{s}{t}$ coordinates, by $$  \frac{10^{4(k-1)}\left(n\cdot \frac{s}{t}\right)\ln^{(k)}\left(n\cdot\frac{s}{t}\right)}{\ln\left(n\cdot\frac{s}{t}\right)}.$$ The maximum degree in $H_3$ is at most $ \frac{500^kn}{s}$ for $n$ sufficiently large. Now note that $\frac{s}{t}=90\frac{\ln^{(k+1)}n}{\ln^{(k)}n}$ and so
\[\sqrt{n}+ \frac{10^{4(k-1)}\left(n\cdot \frac{s}{t}\right)\ln^{(k)}\left(n\cdot\frac{s}{t}\right)}{\ln\left(n\cdot\frac{s}{t}\right)}+ \frac{500^kn}{s}\le 10^{4k}\frac{n\ln^{(k+1)}(n)}{\ln(n)}\]
for $n$ sufficiently large as desired.
\end{proof}
We first prove that it suffices to consider pairs of vertices whose coordinates match along $B_0$ due the presence of the subgraph $H_1$.
\begin{lemma}
If for all vertices $x,y$ with $x$ and $y$ matching along coordinates in $B_0$ we have
\[d_H(x,y)\le d_{Q_n}(x,y)+2k\]
then the same follows for all pairs of vertices $x$ and $y$.
\end{lemma}
\begin{proof}
Consider an arbitrary pair of vertices $x$ and $y$. Let $x'$ such that the first $|B_0|$ coordinates of $x'$ match $y$ and the rest match $x$. Note that  $d_{H}(x,x') = d_{Q_n}(x,x')$ and that $x'$ and $y$ satisfy the condition of the hypothesis. Therefore
\[d_H(x,y) \le d_H(x,x')+d_H(x',y)\le d_{Q_n}(x,x')+d_{Q_n}(x',y) +2k = d_{Q_n}(x,y) +2k\]
and the result follows.
\end{proof}
We now finally prove that $H$ is a $2k$-additive spanner. Recall that we are maintaining the invariant, that for any pair of points which are distance $\ell$ apart there is a path whose length is at most $\ell + 2k$  and whose points have at least $\frac{\ell}{32^{k+1}}$ different coordinate sums.

Furthermore note that the previous lemma does not interfere with this invariant; if the initial points differed in more than $\frac{\ell}{2}$ coordinates in $B_0$ then by applying the above mentioned procedure of first flipping $0$ to $1$ and then $1$ to $0$ we get a path with whose vertices give at least $\frac{\ell}{2}/2=\frac{\ell}{4}$ different coordinate sums which is sufficient. Otherwise, the points differ by at least $\ell/2$ coordinates outside of $B_0$. Therefore it is enough to maintain the invariant that among points with coordinates that match in $|B_0|$ and are of distance $\ell$, there is path of whose vertices give $\frac{\ell}{16\cdot 32^{k}}$ coordinate sums. We now consider two cases.

\textbf{Case 1:} Suppose that $x$ and $y$ differ on at most $s$ coordinates. Then there exist $\{i_1,\ldots,i_s\}$ such that the set of differences is contained inside $B_{i_1},\ldots,B_{i_s}$. Let $D:=f^{-1}(\{i_1,\ldots,i_s\})$. First we go from $x$ and $y$ to  the closest points, call these points $x'$ and $y'$ respectively, in the perfect code $D$ using  edges in $H_1$. Note that since $x$ and $y$ agree in the coordinates in $B_0$, so do $x'$ and $y'$. Now walk from $x'$ to $y'$ using edges of the $2(k-1)$-spanner used to construct $H_2$ and  adjust the necessary bits in $B_{i_1},\ldots,B_{i_s}$. This can be done because $x'$ and $y'$ only differ in coordinates in $B_{i_1},\ldots,B_{i_s}$ and by the definition of $2(k-1)$ spanner the length of this path is by at most $2(k-1)$ larger than their distance in $Q_n$. Therefore, the length of the path that we construct from $x$ to $y$ is at most $d_{Q_n}(x,y)+2(k-1)+2=d_{Q_n}(x,y)+2k$. Moreover we can maintain the invariant regarding coordinate sums by invoking the inductive hypothesis and noticing that $d_{Q_n}(x,y) = d_{Q_n}(x',y')$.

\textbf{Case 2:} Suppose that $x$ and $y$ differ on more than $s$, say $\ell$, coordinates. Then find sets $B_{i_1},\ldots,B_{i_s}$ such that $x$ and $y$ differ by at least $s$ coordinates in $B_{i_1}\cup\ldots\cup B_{i_s}$. We consider two separate situations.
\begin{itemize}
    \item Suppose that $B_{i_1},\ldots,B_{i_s}$ contain more than $\frac{\ell}{5}$ of the coordinate differences. In this case again we  first move from $x$ and $y$  to  the closest points $x'$ and $y'$ (respectively), in the perfect code $D=f^{-1}(i_1,\ldots,i_s)$ using edges in $H_1$.
     Then we use the edges in $H_2$ and the path given by the inductive hypothesis to change coordinates  of $x'$ so that it agrees with $y'$ on $B_{i_1}\cup\ldots \cup B_{i_s}$. 
     
     However when we visit the vertices which give us the  first $\frac{\ell}{10\cdot 32^{k}}$ coordinate sums out of the $\frac{\ell}{5\cdot 32^{k}}$ which is guaranteed by the inductive hypothesis, we keep them to satisfy the inductive hypothesis regarding coordinate sums, and we use the vertices with the remaining $\frac{\ell}{10\cdot 32^{k}}$ coordinate sums to 
   access all the blocks outside of $B_{i_1}\cup\ldots \cup B_{i_s}$ and use the subgraph $H_3$ to fix the differences between $x'$ and $y'$. More precisely, every time we take a step using $H_2$ in a direction in $B_{i_1}\cup\ldots\cup B_{i_s}$ to reach the point whose sum of the coordinates gives us a new residue modulo $s/500^k$ we check if there are any coordinates we can fix using $H_3$, and fix them. As $\frac{\ell}{10\cdot 32^{k}}\ge \frac{s}{500^{k}}$ we will see all the residues modulo $s/500^k$ and so eventually we will be able to fix all the blocks using $H_3$. Note here that we only lost distance 2 in moving to $x'$ and $y'$ and distance $2(k-1)$ in making $x'$ equal to $y'$ on $B_{i_1}\cup\ldots \cup B_{i_s}$. We did not lose any distance anywhere else.

    \item Finally suppose that $B_{i_1},\ldots,B_{i_s}$ contains less than $\frac{\ell}{5}$ of the differences. In this case we have that $x$ and $y$ differ on at least $4\ell/5$ coordinates outside of $B_{i_1}\cup\ldots\cup B_{i_s}$.
   Again, we first move into the points $x'$ and $y'$ in the perfect code $D=f^{-1}(i_1,\ldots,i_s)$ and start fixing the coordinates in $B_{i_1}\cup\ldots\cup B_{i_s}$ using $H_2$ and the path given by the inductive hypothesis. Since $x$ and $y$ differ on at least $s$ coordinates in $B_{i_1}\cup\ldots\cup B_{i_s}$, the vertices of this path have at least $s/32^k$ different coordinate sums. We fix the coordinates outside of $B_{i_1}\cup\ldots\cup B_{i_s}$ in a similar way as in the previous case: after every step we take in $H_2$ we try to fix as many coordinates as possible using edges in $H_3$. However the one difference in this case is, that during the stretch of $\frac{s}{32^{k}}\ge 2\cdot \frac{s}{500^k}$ different coordinate sums we encounter while walking in $H_2$, we use the first half of them to fix all those coordinates where $x$ is zero and $y$ is one, and we use the second half to map all $1$ to $0$. At least one of these halves has length greater than $\frac{2\ell}{5}$, let us assume without loss of generality that it is the half where we change the zeros in $x$ to ones. During these at least $\frac{2\ell}{5}$ steps, the sum of coordinates increases by $+1$ each time. However, we have no control over the steps that we take in $H_2$, they can both increase and decrease the coordinate sum by 1. We have taken at most $\frac{\ell}{5} + 2(k-1)$ steps in $H_2$, so this part of the path has to give at least $\frac{2\ell}{5} - \left(\frac{\ell}{5} + 2(k-1)\right)\geq \frac{\ell}{10}$ different coordinate sums. This allows us to maintain  the desired invariant and walk between  $x$ and $y$ in the required length. 
\end{itemize}
\end{proof}

\section{Concluding remarks and open questions}\label{sec:concl}
The effect of the removal of a set of vertices or edges from computer networks, corresponding to broken connections, processors or inaccessible agents, are of major interest in the study of vulnerability of networks. Parameters that measure changes given by such breakdowns lead to many interesting open problems~\cite{leverage,intsurvey}, from which we only mention a few here.

The integrity $I(Q_n)$ of the hypercube is defined as
$$I(Q_n)=\min\{|S|+m(Q_n\setminus S):S\subset V(Q_n)\},$$
where $m(H)$ denotes the number of vertices in the largest connected component of $H$. It is known (see~\cite{intba,intkl}) that
$$c\frac{2^n}{\sqrt{n}}\leq I(Q_n)\leq C\frac{2^n}{\sqrt{n}}\sqrt{\log n},$$and determining the precise asymptotics would be of interest.

The second problem we have already hinted at in the introduction. 
\begin{question}[Erd\H{o}s--Hamburger--Pippert--Weakley~\cite{Erdos}]\label{ques:fewedges}
What is the least possible number of edges in a graph $G\subset Q_n$ that has diameter $n$?
\end{question}
They observed that there is such a graph $G\subset Q_n$ with $2^n+\binom{n}{\lfloor n/2\rfloor} - 2$ edges: between each pair of consecutive layers $\binom{[n]}{k}$ and $\binom{[n]}{k+1}$ of the hypercube, fix a matching covering the larger layer. We have the following lower bound on~\cref{ques:fewedges}:
\begin{proposition}
If $G\subset Q_n$ has diameter $n$ then $e(G)\geq 2^n + \Theta\left(\frac{2^n}{n}\right)$.
\end{proposition}
\begin{proof}
Let $G\subset Q_n$ be a subgraph with diameter $n$. First we show that $G$ has minimum degree at least 2. Indeed assume $v$ is a leaf, connected only to vertex $u$. Let $u'$ denote the antipodal vertex to $u$. Then the distance of $v$ and $u'$ in $G$ is at least $n+1$, contradicting the fact that $G$ has diameter $n$.

Next, fix an arbitrary vertex $v$ and consider a BFS tree $\mathcal{T}$ rooted at $v$. Since $G$ has diameter $n$, the tree $\mathcal{T}$ has at most $n+1$ layers (the first layer being the single vertex $v$) and hence has a layer $L$ of size at least $\frac{2^{n}-1}{n}$. Note that below each vertex of $L$ there is a distinct leaf, so $\mathcal{T}$ has at least $\frac{2^{n}-1}{n}$ leaves. Since in $G$ every vertex has degree at least two, there must exist at least $\left\lceil\frac{2^{n}-1}{2n}\right\rceil$ edges that are in $G$ but not in $\mathcal{T}$. Hence $$e(G)\geq e(\mathcal{T})+\frac{2^{n}-1}{2n} = 2^n + \frac{2^{n}-1}{2n} - 1.$$
\end{proof}

A question in the same spirit as~\cref{ques:fewedges} concerns 2-additive spanners. Denote by $f_2(n)$ the fewest possible edges in a graph $G\subset Q_n$ that is a 2-additive spanner, that is, $d_G(x,y)\leq d_{Q_n}(x,y)+2$ for all $x,y$. The best known bounds, given in~\cite{balogh20082,Erdos} are
$$c2^n\log n \leq f_2(n)\leq C2^n\sqrt{n}.$$

\bibliographystyle{plain}
\bibliography{hypercube}

\begin{thebibliography}{10}

\bibitem{add5}
Amir Abboud, Greg Bodwin, and Seth Pettie.
\newblock A hierarchy of lower bounds for sublinear additive spanners.
\newblock {\em SIAM Journal on Computing}, 47(6):2203--2236, 2018.

\bibitem{spannersurvey}
Reyan Ahmed, Greg Bodwin, Faryad~Darabi Sahneh, Keaton Hamm, Mohammad
  Javad~Latifi Jebelli, Stephen Kobourov, and Richard Spence.
\newblock Graph spanners: A tutorial review.
\newblock {\em arXiv preprint arXiv:1909.03152}, 2019.

\bibitem{AHK}
Nana Arizumi, Peter Hamburger, and Alexandr Kostochka.
\newblock On k-detour subgraphs of hypercubes.
\newblock {\em Journal of Graph Theory}, 57(1):55--64, 2008.

\bibitem{leverage}
KS~Bagga, LW~Beineke, MJ~Lipman, and RE~Pippert.
\newblock The concept of leverage in network vulnerability.
\newblock {\em Graph theory, combinatorics, and applications}, pages 29--40,
  1991.

\bibitem{intsurvey}
Kunwarjit~S Bagga, Lowell~W Beineke, WD~Goddard, Marc~J Lipman, and Raymond~E
  Pippert.
\newblock A survey of integrity.
\newblock {\em Discrete Applied Mathematics}, 37:13--28, 1992.

\bibitem{balogh20082}
J{\'o}zsef Balogh and Alexandr Kostochka.
\newblock On 2-detour subgraphs of the hypercube.
\newblock {\em Graphs and Combinatorics}, 24(4):265--272, 2008.

\bibitem{intba}
J{\'o}zsef Balogh, Tam{\'a}s M{\'e}sz{\'a}ros, and Adam~Zsolt Wagner.
\newblock Two results about the hypercube.
\newblock {\em Discrete Applied Mathematics}, 247:322--326, 2018.

\bibitem{intkl}
Lowell~W Beineke, Wayne Goddard, Peter Hamburger, Daniel~J Kleitman, Marc~J
  Lipman, and Raymond~E Pippert.
\newblock The integrity of the cube is small.
\newblock {\em J. Combin. Math. Combin. Comput}, 9:191--193, 1991.

\bibitem{add1}
B{\'e}la Bollob{\'a}s, Don Coppersmith, and Michael Elkin.
\newblock Sparse distance preservers and additive spanners.
\newblock {\em SIAM Journal on Discrete Mathematics}, 19(4):1029--1055, 2005.

\bibitem{add6}
Gilad Braunschvig, Shiri Chechik, David Peleg, and Adam Sealfon.
\newblock Fault tolerant additive and ($\mu$, $\alpha$)-spanners.
\newblock {\em Theoretical Computer Science}, 580:94--100, 2015.

\bibitem{add3}
Keren Censor-Hillel, Telikepalli Kavitha, Ami Paz, and Amir Yehudayoff.
\newblock Distributed construction of purely additive spanners.
\newblock {\em Distributed Computing}, 31(3):223--240, 2018.

\bibitem{add2}
Michael Elkin and Shaked Matar.
\newblock Near-additive spanners in low polynomial deterministic congest time.
\newblock {\em arXiv preprint arXiv:1903.00872}, 2019.

\bibitem{Erdos}
P{\'a}l Erd{\H{o}}s, Peter Hamburger, Raymond~E Pippert, and William~D Weakley.
\newblock Hypercube subgraphs with minimal detours.
\newblock {\em Journal of Graph Theory}, 23(2):119--128, 1996.

\bibitem{perfectcodes}
Florence~Jessie MacWilliams and Neil James~Alexander Sloane.
\newblock {\em The theory of error-correcting codes}, volume~16.
\newblock Elsevier, 1977.

\bibitem{add7}
Merav Parter.
\newblock Vertex fault tolerant additive spanners.
\newblock {\em Distributed Computing}, 30(5):357--372, 2017.

\bibitem{add4}
David~P Woodruff.
\newblock Lower bounds for additive spanners, emulators, and more.
\newblock In {\em 2006 47th Annual IEEE Symposium on Foundations of Computer
  Science (FOCS'06)}, pages 389--398. IEEE, 2006.

\end{thebibliography}

\end{document}